\documentclass[11pt]{article}
\usepackage{latexsym}
\usepackage{epsfig,enumerate}
\usepackage{amsmath,amsthm,amssymb, bbm}
\usepackage[usenames]{color}\usepackage{graphicx}
\numberwithin{equation}{section}
\definecolor{brown}{cmyk}{0, 0.72, 1, 0.45}
\definecolor{grey}{gray}{0.5}

\renewcommand{\epsilon}{\varepsilon}

\newcounter{rot}

\def\a{\alpha}  \def\d{\delta} \def\D{\Delta}
\def\e{\epsilon} \def\f{\phi}   \def\g{\gamma}
\def\G{\Gamma}  \def\k{\kappa} 
\def\z{\zeta}     
   
  \def\s{\sigma} 
\def\t{\tau}

\newtheorem*{conjecture*}{Conjecture}
\newtheorem{theorem}{Theorem}[section]
\newtheorem*{theorem*}{Theorem}
\newtheorem{lemma}[theorem]{Lemma}

\def\es{\emptyset}

\def\E{\mathbb{E}}

\def\P{\mathbb{P}}
\def\Pr{\mathbb{P}}

\newcommand{\bfo}{\mathbbm{1}}
\newcommand{\scr}{\mathcal}

\newcommand{\of}[1]{\left( #1 \right) }

\newcommand{\abs}[1]{\left| #1 \right|}

\newcommand{\sqbs}[1]{\left[ #1 \right]}
\newcommand{\braces}[1]{\left\{ #1 \right\}}
\newcommand{\Mean}[1]{\E\sqbs{#1}}

\renewcommand{\Pr}[1]{\mathbb{P}\left[ #1 \right]}

\newcommand{\reas}[1]{\textrm{re-assign}(#1)}

\allowdisplaybreaks[1]
\newcommand{\ignore}[1]{}

\newcommand{\beq}[1]{\begin{equation}\label{#1}}
\newcommand{\eeq}{\end{equation}}

\newcommand{\Bin}{\operatorname{Bin}}

\def\VG{V_\G}

\title{A greedy algorithm for finding a large 2-matching on a random cubic graph}

\author{Deepak Bal\and Patrick Bennett\and Tom Bohman\thanks{Research supported in part by NSF Grant DMS1001638} \and Alan Frieze\thanks{Research supported in
part by NSF Grant CCF1013110}\\Department of Mathematical Sciences\\Carnegie Mellon University\\Pittsburgh PA15213\\USA}

\begin{document}
\maketitle
\begin{abstract}
A 2-matching of a graph $G$ is a spanning subgraph with maximum degree two. The size of a 2-matching $U$ is the number of edges in $U$ and this is at least $n-\k(U)$ where $n$ is the number of vertices of $G$ and $\k$ denotes the number of components. In this paper, we analyze the performance of a greedy algorithm \textsc{2greedy} for finding a large 2-matching on a random 3-regular graph. We prove that with high probability, the algorithm outputs a 2-matching $U$ with $\k(U) = \tilde{\Theta}\of{n^{1/5}}$.
\end{abstract}

\section{Introduction}
In this paper we analyze the performance of a generalization of the well-known Karp-Sipser algorithm \cite{KS,FRS,AFP,BF} for finding a large matching
in a sparse random graph. A 2-matching $U$ of a graph $G$ is a spanning subgraph with maximum degree two.  Our aim is to show that w.h.p. our algorithm finds a large 2-matching in a random cubic graph. 
 The algorithm \textsc{2greedy} is described below and has been partially
analyzed on the random graph $G_{n,cn}^{\d\geq 3},\,c\geq 10$ in Frieze \cite{F11}. The random graph $G_{n,m}^{\d\geq3}$ is chosen uniformly at random from the collection of all graphs that have $n$ vertices, $m$ edges and minimum degree $\d(G)\geq 3$.  In \cite{F11}, the 2-matching output by the algorithm is used to find a  Hamilton cycle in $O(n^{1.5 + o(1)})$ time w.h.p.  Previously, the best known result  for this model was that $G_{n,cn}^{\d\ge 3}$ is Hamiltonian for $c \ge 64$ due to Bollob\'{a}s, Cooper, Fenner and Frieze \cite{BCFF}. It is conjectured that $G_{n,cn}^{\d\ge3}$ is Hamiltonian w.h.p. for all $c\ge 3/2$.

The existence of Hamilton cycles in other random graph models with $O(n)$ edges has also been the subject of much research. In such graphs, the requirement $\d \ge 3$ is necessary to avoid three vertices of degree two sharing a common neighbor. This obvious obstruction occurs with positive probability in many models with $O(n)$ edges and $\d=2$.
$G_{\textrm{3-out}}$ is a random graph where each vertex chooses 3 neighbors uniformly at random. This graph has minimum degree 3 and average degree 6. Bohman and Frieze proved that $G_{\textrm{3-out}}$ is Hamiltonian w.h.p. also by building a large 2-matching into a Hamilton cycle \cite{BF11}.
 Robinson and Wormald proved that $r$-regular graphs with $r\ge 3$ are Hamiltonian w.h.p. using an intricate second moment approach \cite{RW92},\cite{RW94}. Before this result, Frieze proved Hamiltonicity of $r$-regular graphs w.h.p. for $r\ge 85$ using an algorithmic approach \cite{F88}. An algorithmic proof of Hamiltonicity for $r\ge 3$ was given in \cite{FJMRW}.

In addition to the Hamiltonicity of $ G_{n,cn}^{\delta \ge 3}$ for $ 3/2 < c < 10 $, the Hamiltonicity of random graphs with $ O(n) $ edges and a 
 fixed degree sequence is a widely open question.  One natural example is the Hamiltonicity of a graph chosen uniformly at random from all the collection of all graphs with $ n/2 $ vertices of degree 3 and $n/2$ vertices of degree 4 (this particular question was posed by Wormald).  For both $ G_{n,cn}^{\delta \ge 3} $ and graphs with a fixed degree sequence one might hope to prove Hamiltonicity by first using \textsc{2greedy} to produce a large 2-matching and then using an extension rotation argument to convert this 2-matching into a Hamilton cycle.  In this paper we provide evidence that the first half of this broad program is feasible by showing that \textsc{2greedy} finds a very large 2-matching for the sparsest of the models with minimum degree 3, the random cubic graph itself.

The size of a 2-matching $U$ is the number of edges in $U$ and this is at least $n-\k(U)$ where $\k$ denotes the number of components. It was shown in 
\cite{FRS} that w.h.p. the Karp-Sipser algorithm only leaves $\tilde{\Theta}(n^{1/5})$ vertices unmatched. Here we prove the corresponding result 
for \textsc{2greedy} on a random cubic graph.

\begin{theorem}\label{thm:main}
Algorithm \textsc{2greedy} run on a random 3-regular graph with $n$ vertices outputs a 2-matching $U$ with $\k(U)=\tilde{\Theta}(n^{1/5})$, w.h.p.
\end{theorem}
\noindent We prove Theorem~\ref{thm:main} using the differential equations method for establishing dynamic concentration.  The remainder of the paper is organized as follows.  The \textsc{2greedy} algorithm is introduced in the next Section, and the random variables we track are given in Section~3.  The trajectories that we expect these variables to follow are given in Section~4.   A heuristic explanation of why \textsc{2greedy} should produce a 2-matching with roughly $ n^{1/5} $ components is also given in Section~4.  In Section~5 we state and prove our dynamic concentration result. The proof of Thereom~\ref{thm:main} is then completed in Sections~5, 6, and~7.

\section{The Algorithm}
The Karp-Sipser algorithm for finding a large matching in a sparse random graph is essentially the greedy algorithm, with one slight modification that makes a big difference. While there are vertices of degree one in the graph, the algorithm adds to the matching an edge incident with such a vertex. Otherwise, the algorithm chooses a random edge to add to the matching. The idea is that no mistakes are made while pendant edges are chosen since such edges are always contained in some maximum matching. The algorithm presented in \cite{F11} is a generalization of Karp-Sipser for 2-matchings. Our algorithm is essentially the same as that presented in \cite{F11} applied to random cubic graphs.  A few slight modifications have been made to 
ease the analysis and to account for the change in model. We assume that our input (multi-)graph $G = G([n],E)$ is generated by the configuration 
model of Bollobas \cite{Bol}. Let $W=[3n]$ be our set
of {\em configuration points} and let $W_i=[3(i-1)+1,3i]$,
$i\in [n]$, partition $W$. The function $\f:W\to[n]$ is defined by
$w\in W_{\f(w)}$. Given a
pairing $F$ (i.e. a partition of $W$ into $m=3n/2$ pairs) we obtain a
(multi-)graph $G_F$ with vertex set $[n]$ and an edge $(\f(u),\f(v))$ for each
$\{u,v\}\in F$. Choosing a pairing $F$ uniformly at random from
among all possible pairings $\Omega$ of the points of $W$ produces a random
(multi-)graph $G_F$. It is known that conditional on $G_F$ being simple, i.e. having no loops or multiple edges, that it is equally 
likely to be any (simple) cubic graph. Further, $G_F$ is simple with probability $(1-o(1))e^{-2}$. So from now on we work with $G=G_F$.

We only reveal adjacencies (pairings) of $G_F$ as the need arises in the algorithm. As the algorithm 
progresses, it grows a 2-matching and deletes vertices and edges from the input graph $G$. We let $\G=(V_\G,E_\G)$ be the current state of $G$. Throughout the algorithm 
we keep track of the following:
\begin{itemize}
 \item $U$ is the set of edges of the current 2-matching. The internal vertices and edges of the paths and cycles in $U$ will have been deleted from $\G$.
 \item $b(v)$ is the 0-1 indicator for vertex $v\in[n]$ being adjacent to an edge of $U$. 
 \item $Y_k = \braces{v\in \VG : d_\G(v)=k,\,b(v)=0}$, $k=0,1,2,3.$
 \item $Z_k = \braces{v\in\VG : d_\G(v)=k,\,b(v)=1}$, $k=0,1,2.$
\end{itemize}

We refer to the sets $Y_3$ and $Z_2$ as $Y$ and $Z$ throughout. The basic idea of the algorithm is as follows.  We add edges to the 2-matching one by one, which sometimes forces us to delete edges. These deletions may put vertices in danger of having degree less than 2 in the final 2-matching. Thus, we prioritize the edges that we add to $U$, so as to match the dangerous vertices first.  More precisely, At each iteration of the algorithm, a vertex $v$ is chosen and an adjacent edge is added to $U$.  We choose $v$ from the first non-empty set in the following list: $Y_1, Y_2, Z_1, Y, Z$.  As in the Karp-Sipser algorithm, taking edges adjacent to the vertices in $Y_1$, $Y_2$ and $Z_1$ is not a mistake. 
We will prove that by proceeding in this manner, we do not create too many components. 

When a vertex $v$ is chosen and its neighbor in the configuration is exposed it is called
a $\emph{selection move}$. Call the revealed neighbor, $w$ the \emph{selection}. The edge $(v,w)$ is removed from $\G$ and added to $U$. 
If the selction $w$ is a vertex in $Z$, then once $(v,w)$ is added to $U$, we must delete the other edge adjacent to $w$. Hence we reveal the other edge in the 
configuration adjacent to $w$. Call this exposure a \emph{deletion move}. 

Details of the algorithm are now given.

\noindent\textbf{Algorithm} \textsc{2Greedy}:

Initially, all vertices are in $Y$. Iterate the following steps as long as one of the conditions holds.

\begin{enumerate}
 \item[]{\bf Step 1(a)} $Y_1\ne\emptyset$.

Choose a random vertex $v$ of $Y_1$. Suppose its neighbor in $\G$ is $w$. Remove $(v,w)$ from $\G$ and add it to $U$. Set $b(v)=1$ and move $v$ to $Z_0$. 

$\reas{w}.$

\item[]{\bf Step 1(b)} $Y_1=\es, Y_2\ne \es$.

Choose a random vertex $v$ of $Y_2$. Randomly choose one of the two neighbors of $v$ in $\G$ to expose and call it $w$.  

If $w=v$ ($\braces{v}$ comprises an isolated component in $\G$ with a loop), then remove $(v,v)$ from $\G$ and move $v$ from $Y_2$ to $Y_0$.

Otherwise, remove $(v,w)$ from $\G$ and add it to $U$.  Set $b(v)=1$ and move it to $Z_1$.

$\reas{w}$.

\item[]{\bf Step 1(c)} $Y_1=Y_2=\es, Z_1\ne \es$.

Choose a random vertex $v$ of $Z_1$. $v$ is the endpoint of a path in $U$. Let $u$ be the other endpoint of this path. Suppose the neighbor of $v$ in $\G$ is $w$. Remove $(v,w)$ from $\G$ and add 
it to $U$. Remove $v$ from $\G$. 

$\reas{w}$.

\item[]{\bf Step 2} $Y_1=Y_2=Z_1= \es, Y\ne\es$.

Choose a random vertex $v$ of $Y$. Randomly choose one of the three neighbors of $v$ in $\G$ to expose and call it $w$.  

If $w=v$, then we remove $(v,v)$ from $\G$ and move $v$ to $Y_1$.

Otherwise, remove $(v,w)$ from $\G$ and add it to $U$. Set $b(v)=1$ and move it to $Z$.

$\reas{w}$.

\item[]{\bf Step 3} $Y_1=Y_2=Z_1=Y=\es, Z\ne\es$

The remaining graph is a random 2-regular graph on $\abs{Z}$ many vertices. Put a maximum matching on the remaining graph. Add the edges of this matching to $U$.
\end{enumerate}

\noindent\textbf{Subroutine} $\reas{w}$:
\begin{enumerate}
 \item If $b(w)=0$: 

Set $b(w)=1$ and move $w$ from $Y$ to $Z$, $Y_2$ to $Z_1$ or $Y_1$ to $Z_0$ depending 
on the initial state of $w$.
\item If $b(w)=1$: 

Remove $w$ from $\G$. If $w$ was in $Z$ prior to removal, then the removal of $w$ from $\G$ causes an edge $(w,w')$, to be deleted from $\G$.  Move $w'$ to the appropriate new set. For example, if $w'$ were in $Z$, it would be moved to $Z_1$; if $w'$ were in $Y$, it would be moved to $Y_2$, etc.

\end{enumerate}

\section{The Variables}
In this section we will describe the variables which are tracked as the algorithm proceeds. Throughout the paper, in a slight abuse of notation, we let $Y, Z$, etc. 
refer to both the sets and the size of the set. Let $M$ refer to the size of $E_\G$. We also define the variable 
\[\z := Y_1 + 2Y_2 + Z_1.\]
If $X$ is a variable indexed by $i$, we define
\[\D X(i) := X(i+1) - X(i).\]
\subsection{The sequences $\s, \d$}

We define two sequences $\s, \d$ indexed by the step number $i$. $\s(i)$ will indicate what type of vertex is selected during a selection move, and $\d(i)$ will do 
the same for deletion moves. 

Formally, $\s$ is a sequence of the following symbols: $Y, Z, \z, loop, multi$.
We will put $\s(i)=loop$ only when step $i$ is of type $2$ and the selection move reveals a loop. 
We put $\s(i) = multi$ only when step $i$ is of type $1(c)$, and $w = u \in Z$. The only way this happens is when $v \in Z_1$,  $u \in Z$, $(v, u) \in U$, and the 
selection made at step $i$ happens to select the vertex $u$. 
Otherwise we just put $\s(i) = Y, Z, \z$ according to whether the selected vertex is in $Y, Z, \z$. 

Note that the symbols $loop, multi$ are for very specific events, and not just any loop or multi-edge. If step $i$ is of type $1(b)$ and our selection move reveals a 
loop, then we put $\s(i)=\z$. Also, if step $i$ is of type $1(c)$ and the selection move reveals a multi-edge whose other endpoint is also in $Z_1$ then we put 
$\s(i) = \z$ as well.

$\d$ is a sequence of symbols: $Y, Z, \z, \emptyset$.  We will put $\d(i) = \emptyset$ when there is no deletion move at step $i$ (i.e. when $\s(i) \notin \{Z, multi\}$). Otherwise $\d(i)$ just indicates the type of vertex that the deletion move picks (here we don't make any distinctions regarding loops or 
multi-edges).

\subsection{The variables $A,B$}

We will define the following two important variables:

$$A := Y + \zeta $$
$$B := 2Y + Z + \zeta.$$
$A$ is a natural quantity to define, since the algorithm terminates precisely when $A=0$. $B$ is also natural because it represents the number of half-edges which 
will (optimistically) be added to our current $2$-matching before termination.  We will see that $A$ and $B$ are also nice variables in that their $1$-step changes $\Delta A(i), \Delta B(i)$ do not depend on what type of step we 
take at step $i$. 
We have
\begin{align}
\Delta Y(i) &= -\bfo_{\zeta(i)=0} - \bfo_{\sigma(i)=Y} - \of{\bfo_{\s(i)=Z} + \bfo_{\s(i)=multi}} \bfo_{ \delta(i) = Y}  \\
\Delta Z(i) &= \bfo_{\zeta(i)=0} + \bfo_{\sigma(i)=Y} - \bfo_{\sigma(i)=Z} -\bfo_{\s(i)=loop}-\bfo_{\s(i)=multi}\nonumber \\
& \quad - \of{\bfo_{\s(i)=Z} + \bfo_{\s(i)=multi}}\bfo_{\delta(i)=Z} \\
\Delta \zeta(i) &= -\bfo_{\zeta(i)>0} +\bfo_{\s(i)=loop} - \bfo_{\sigma(i) = \zeta} \nonumber\\
&\quad +\of{\bfo_{\s(i)=Z}+ \bfo_{\s(i)=multi}}\of{ - \bfo_{\delta(i)=\zeta}+  \bfo_{\delta(i) = Z} + 2 \cdot \bfo_{\delta(i) = Y}} \label{1sczeta}
\end{align}
and note that these all depend on whether $\zeta=0$ (i.e. whether step $i$ is of type $1$ or $2$). However,
\begin{align}
\Delta A(i) &= -1 -\bfo_{\sigma(i)=Y} - \bfo_{\sigma(i) = \zeta}+\bfo_{\s(i)=loop}-\bfo_{\s(i)=multi} + \bfo_{\sigma(i)=Z} \nonumber\\
  & \quad-\of{\bfo_{\s(i)=Z} + \bfo_{\s(i)=multi}} 2 \cdot \bfo_{\delta(i)=\zeta} \label{1scA}\\
\Delta B(i) &= -2 +\bfo_{\s(i)=loop}- \bfo_{\delta(i)= \zeta} \label{1scB}
\end{align}
which do not depend on whether $\zeta=0$. For $\D A$, we have used the identity \[\bfo_{\d=Y} + \bfo_{\d=Z}+\bfo_{\d=\z} = \bfo_{\s=Z} + \bfo_{\s=multi}\] which 
states that we make a deletion move if and only if our selection move was $Z$ or $multi$.
Note also that if we establish dynamic concentration on $A, B, \zeta$ then we implicitly establish concentration on $Y, Z, M$ since 
\begin{align}
Y &=  A - \zeta \\
Z &= B-2A + \zeta \\
2M &= 3Y + 2Z + \zeta = 2B-A.
\end{align}

\section{The expected behavior of $A, B, \zeta$}\label{sec:expbehavior}
In this section, we we will non-rigorously predict the behavior of the variables and some facts about the process. 
Throughout the paper, unless otherwise specified, $t$ refers to the scaled version of $i$, so \[t: =\frac{i}{n}.\] Heuristically, we assume there exist 
differentiable functions $a, b$ such that $A(i) \approx na(t), B(i) \approx nb(t)$. Further, we assume that $\zeta$ stays ``small''. We will prove that these 
assumptions are indeed valid. We also let
\[p_z := \frac{2Z}{2M},\quad p_y := \frac{3Y}{2M},\quad p_\z := \frac{\z}{2M}\]
where we have omitted the dependence on $i$ for ease of notation.

\subsection{The trajectory $b(t)$}

Since $B(0)=2n$, and recalling $\eqref{1scB}$, we see that
\begin{equation} \label{eq:Btrajectory}
B(i) = 2n - 2i + \sum_{j \le i} \of{1_{\s(j)=loop}-  1_{\delta(j)=\zeta} }.\end{equation}
The probability that $\s(j)=loop$ or $\d(j)=\z$ on any step $j$ should be negligible.  Thus we expect
$$B(i) \approx 2n-2i = 2n(1-t)$$
so we will set $$b(t) = 2(1-t).$$

\subsection{The trajectory $a(t)$}

We derive an ODE that $a$ should satisfy:
\[a'(t) \approx E[ \Delta A(i) ] \approx -1 -p_y + p_z \approx -\frac{6a(t)}{2b(t)-a(t)}.\]
Note that we have used $\eqref{1scA}$. Thus $a(t)$ should satisfy
\begin{align}
a' = -\frac{6a}{4-4t-a}. \label{diffeq}
\end{align}
The substitution $x=\frac{a}{1-t}$ yields a separable ODE,  which can be integrated to arrive at 
$$(a+2-2t)^3 -27a^2=0.$$
So $a(t)$ is given implicitly as the solution to the above cubic equation with coefficients depending on $t$. We may actually solve that cubic to get three 
continuous explicit functions $a_1(t), a_2(t), a_3(t)$ (though the formulas are nasty to look at). From the initial condition $a(0)=1$ and the fact that 
$0 \leq a(t) \leq 1$, it's clear that the solution we want is

\[a(t) = 7+2t-6 \sqrt{5+4t} \cos\left(\frac{1}{3} \arccos \left( \frac{11+14t+2t^2}{(5+4t)^\frac{3}{2}} \right) + \frac{\pi}{3} \right).\]
 From here we can see that $a(t) \rightarrow 0$ as $t \rightarrow 1^-$. More precisely,

\begin{align}\label{eq:alimit}
&\displaystyle \lim_{t \rightarrow 1^-} \frac{a(t)}{(1-t)^\frac{3}{2}} = \left(\frac{2}{3} \right)^\frac{3}{2}.
\end{align}
To confirm this, note that \[\arccos\of{1-\e} = \sqrt{2\e} + O(\e^{3/2})\] and
\[\frac{11 + 14(1-\e) + 2(1-\e)^2}{\of{5+4(1-\e)}^{3/2}} = 1-\frac{4\e^3}{729} + O(\e^4).\]
Rewriting the $\cos$ term using the angle addition formula and Taylor expansion, we see $\eqref{eq:alimit}$. Additionally,
\begin{align}
 \frac{d}{dt}\of{\frac{a(t)}{(1-t)^{3/2}}} &= -\frac{6a}{2b-a}(1-t)^{-3/2} + \frac{3}{2}a\cdot(1-t)^{-5/2}\\
 &=a\cdot(1-t)^{-5/2}\of{\frac{3}{2}-\frac{6(1-t)}{4(1-t) - a}}\\
 &< 0.
\end{align}
Since $a(0)=1$, for all $0\le t \le 1$ we have 
\begin{equation}\label{eq:abounds}
 \of{\frac{2}{3}}^{\frac{3}{2}}(1-t)^{3/2} \le a(t) \le (1-t)^{3/2}.
\end{equation}

\subsection{Downward drift of $\z$}

We expect $\z$ to be ``small'', and to justify that claim we will show that whenever $\z$ is positive, it is likely to decrease. Assume that $\z(i)>0$. In the 
following table, we make use of the fact that $\d(i) \neq\es$ if and only if $\s(i) \in \braces{Z,multi}$. So for example, \[\bfo_{\d=Y} = (\bfo_{\s=Z} + 
\bfo_{\s=multi})\bfo_{\d=Y}.\] 
Then from $\eqref{1sczeta}$ we see that if $\zeta(i)>0$,

\begin{equation}\label{zetatable}
\Delta \zeta = \begin{cases}
1&\mbox{with prob. } p_zp_y+O\of{\frac{1}{M}} \\
0&\mbox{with prob. } p_z^2+O\of{\frac{1}{M}} \\
-1&\mbox{with prob. } p_y+O\of{\frac{1}{M}}   \\
-2&\mbox{with prob. } p_\z+p_zp_\z+O\of{\frac{1}{M}} 
\end{cases}
\end{equation}
To illuminate this table, we provide an example. 
\[\Pr{\Delta\z(i) = 1} = \Pr{\s(i)=Z, \, \d(i)=Y} = \frac{2Z}{2M-1}\frac{3Y}{2M-3} = p_zp_y + O\of{\frac{1}{M}}.\]
Therefore, roughly speaking we have
\begin{align}
E[ \Delta \zeta (i) ] &= p_z p_y - p_y + O \of{ p_\z} \approx -\frac{9a^2}{(2b-a)^2} 
\end{align}
and are motivated to define
$$\Phi(t) := \frac{9a^2}{(2b-a)^2} = \Theta(1-t)$$
to represent the downward drift of $\z(i)$ (if it is positive) at step $i$.

\subsection{Expected behavior of $\z$}

In the last subsection we estimated $E[ \Delta \zeta (i) ]$ when $\z>0$, using $\eqref{zetatable}$. We can also use $\eqref{zetatable}$ to estimate the variance when $\z>0$. We see that 
$$Var[ \Delta \zeta (i) | \z>0] = \Theta(p_y) = \Theta\of{ (1-t)^\frac{1}{2}}.$$
Thus, to model the behavior of $\z(i)$ we consider a simpler variable: a lazy random walk $X_\t(k)$ with $X_\t(0)=0$, expected 1-step change $\Mean{\Delta X_\t}=-(1-\t)$ and $Var[\Delta X_\t]=(1-\t)^\frac{1}{2}$. After $s$ steps, we have $ \Mean{X_\t(s)}=-(1-\t)s$ and $Var[X_\t(s)]=(1-\t)^\frac{1}{2}s$. 
There is at least constant (bounded away from $0$) probability that $X_\t(s)$ is, say, 1 standard deviation above its mean. However, the probability that $X_\t(s)$ is very many standard deviations larger than that is negligible. In other words, it is reasonable to have a displacement as large as $X_\t(s) = -(1-\t)s + (1-\t)^\frac{1}{4}s^\frac{1}{2}$, but not much larger.
The quantity $\psi(s):=-(1-\t)s + (1-\t)^\frac{1}{4}s^\frac{1}{2}$ is negative for $s > (1-\t)^{-\frac{3}{2}}$. Also $\psi(s)$ is maximized when $s= \frac{1}{2} (1-\t)^{-\frac{3}{2}}$, where we have $\psi(s) = \frac{1}{4} (1-\t)^{-\frac{1}{2}}$.

Now we reconsider the variable $\z$. Roughly speaking, $\z(i)$ behaves like the lazy random walk considered above, so long as we restrict the variable $i$ to a short range (so that $t$ does not change significantly), and we have $\z(i)>0$ for this range of $i$. We have $\z(0)=0$, and $\z$ has a negative drift so it's likely that $\z(j)=0$ for many $j>0$. Specifically, if $j$ is an index such that $\z(j)=0$, then we expect $\z(i)$ to behave like $X_{\t}(i-j)$ with $\t=\frac{j}{n}$, so long as $i$ is not significantly larger than $j$. Thus we expect to have $\z(i)=0$ for some $j \le i \le j+(1-\t)^{-\frac{3}{2}}$. Also, for all $j \le i \le j+(1-\t)^{-\frac{3}{2}}$ we should have $\z(i) \le  \frac{1}{4} (1-\t)^{-\frac{1}{2}}$.
But this rough analysis does not make sense toward the end of the process: indeed, for $j > n - n^\frac{3}{5}$ (i.e. for $1-\t < n^{-\frac{2}{5}}$), we have $j+(1-t)^{-\frac{3}{2}} > n$. However, we can still say something about what happens when $j$ is large, since the variable $s$ cannot be any bigger than $n-j$. Now for $j \ge n - n^\frac{3}{5}$ and $s \le n-j$ we have $\psi(s) \le n^\frac{1}{5}$. Thus, we never expect $\z$ to be larger than $n^{\frac{1}{5}}$, even towards the end of the process.

\subsection{Why do we have $\tilde{\Theta}\of{n^{\frac{1}{5}}}$ many components?}

At any step of the algorithm, we expect the components of the $2$-matching to be mostly paths (and a few cycles). We would like the algorithm to keep making the paths longer, but sometimes it isn't possible to make a path any longer because of deletion moves. Specifically, for example, if one endpoint of a path is in $Z_1$, and then there is a deletion move and the deletion is that endpoint, then that end of the path will never grow. If the same thing happens to the other endpoint of the path, then the path will never get longer, and will never be connected to any of the other paths. Similarly, the number of components in the final $2$-matching is increased whenever the algorithm deletes a $Y_1$ or a $Y_2$. Thus we can bound the number of components in the final $2$-matching by bounding the number of steps $i$ such that $\d(i) = \z$.

Roughly, $\P [\d(i) = \z] = \frac{2Z}{2M} \cdot \frac{\z}{2M} =  \Theta \of{\frac{1}{n} \min \left\{(1-t)^{-\frac{3}{2}} , \frac{n^\frac{1}{5}}{1-t}  \right\} }$. So integrating, we estimate the total number of components as $$\Theta\of{  \int_0^{1-\frac{1}{n}} \min \left\{(1-t)^{-\frac{3}{2}} , \frac{n^\frac{1}{5}}{1-t} \right\} dt} = \Theta \of{n^{\frac{1}{5}} \log n}.$$

\section{The stopping time $T$ and dynamic concentration}

In this section, we introduce a stopping time $T$, before which $A$ and $B$ stay close to their trajectories, and $\zeta$ does roughly what we expect it to do. 
We will also introduce ``error'' terms for both $A, B$ and a ``correction'' term $\a$ for the variable $A$. For most of the process, $\a$ will 
stay smaller than the error term for $A$. However, toward the end of the process $\a$ will be significant. Using $\a$ in our calculations thus allows us to track 
the process farther. As it turns out, the variable $B$ does not need an analogous ``correction'' term. 

We define the following random variables which represent ``actual error'' in $A, B$:
$$e_a(i) := A(i) - na(t) - \a(i)$$
$$e_b(i) := B(i) - nb(t).$$
We define the stopping time $T$ as the minimum of $n - C_T n^\frac{7}{15} \log^\frac{6}{5} n$ and the first step $i$ such that any of the three following conditions 
fail:
\begin{align}
&|e_a(i)| \le f_a\of{t} \label{ea}, \\
& |e_b(i)| \le f_b\of{t} \label{eb},
\end{align}
and for every step $j<i$ such that $\zeta$ is positive on steps $j, \ldots, i$, 
 \begin{align}
 \zeta(i) \le \zeta(j) -\displaystyle \sum_{j \le k < i} \Phi \left(\frac{k}{n} \right)+ \ell_j \of{t } \label{ezeta}
 \end{align}
for some as-yet unspecified error functions $f_a, f_b, \ell_j$ and absolute constant $C_T$. Throughout the paper we will use $C_?$ to refer to unspecified but 
existent absolute constants. In subsection \ref{sec:kappavalues}, we present actual values for these constants.

We anticipate that the conditions on $\zeta$ will  imply that for some function $f_\zeta$ we have 
$$\zeta(i) \le f_\zeta \of{t}$$
for all $i \le T$.
Our goal for now is to prove that for some suitable error functions, w.h.p. $T$ is not triggered by any of the conditions $\eqref{ea}, \eqref{eb}, \eqref{ezeta}$.

\begin{theorem} \label{Ttheorem}
With high probability, $$T =n - C_T n^\frac{7}{15} \log^\frac{6}{5} n.$$
\end{theorem}
The remainder of this section contains the proof of Theorem $\ref{Ttheorem}$. Here we define the error functions $f_a,f_b, f_\z$ (up to the choice of constants). While these definitions are not very enlightening at this point, they will aid the reader in confirming many of the calculations that appear below. Those same calculations will motivate the choice of these functions.
\begin{align}
 f_a(t) &:=  C_A (1-t)^\frac{3}{4} n^\frac{1}{2} \log^\frac{1}{2} n \\
 f_b(t) &:= C_B \cdot \begin{cases}
 (1-t)^{-\frac{1}{2}} \log n&: 1-t > n^{-\frac{2}{5}} \log^\frac{2}{5} n \\
-n^\frac{1}{5} \log^\frac{4}{5} n \log(1-t)&: \textrm{otherwise} \\
\end{cases}\\
f_\zeta (t) &:= C_\z \min \left\{ (1-t)^{-\frac{1}{2}} \log n, n^\frac{1}{5} \log^\frac{4}{5} 
n \right\}.
\end{align}

\subsection{A useful lemma}

We'll use the following simple lemma several times to estimate fractions.

\begin{lemma} \label{estlemma}

For any real numbers $x, y, \epsilon_x, \epsilon_y$, if we have $x,y \neq 0$ and $\left|\frac{\epsilon_x}{x}\right|, \left|\frac{\epsilon_y}{y}\right| 
\le \frac{1}{2}$, then

$$\frac{x+ \epsilon_x}{y+\epsilon_y} - \frac{x}{y} = \frac{y \epsilon_x - x \epsilon_y}{y^2} + O\left(\frac{y \epsilon_x \epsilon_y + x \epsilon_y^2}{y^3} \right)$$

\end {lemma}

\begin{proof}
\begin{align*}
\frac{x+ \epsilon_x}{y+\epsilon_y} - \frac{x}{y} &= \frac{x}{y} \left\{ \left(1+\frac{\epsilon_x}{x} \right) \cdot \frac{1}{1+\frac{\epsilon_y}{y}} -1 \right\}\\
&= \frac{x}{y} \left\{ \left(1+\frac{\epsilon_x}{x} \right) \cdot \left[1 - \frac{\epsilon_y}{y} + O\left(\frac{\epsilon_y^2}{y^2} \right) \right] -1 \right\}\\
&= \frac{x}{y} \left\{ \frac{\epsilon_x}{x} - \frac{\epsilon_y}{y} + O\left(\frac{\epsilon_x \epsilon_y}{x y}+\frac{\epsilon_y^2}{y^2} \right)\right\}\\
&= \frac{y \epsilon_x - x \epsilon_y}{y^2} + O\left(\frac{y \epsilon_x \epsilon_y + x \epsilon_y^2}{y^3} \right)
\end{align*}
\end{proof}

\subsection{$T$ is not triggered by $A$}

We define 

$$A^+ (i) := A(i) - n  a(t) -\a(i) - f_a(t) = e_a(i) - f_a(t)$$
and let the stopping time $T_j$ be the maximum of $j$, $T$, and the least index $i \geq j$ such that $e_a(i)$ is not in the critical interval 
\begin{equation}\label{eq:acrit}
 [g_a(t),f_a(t)]
\end{equation}
where $0< g_a < f_a$ is an as-yet unspecified function of $n, t$. Our strategy is to show that w.h.p. $A$ never goes above $na+ \a+f_a$ because every time $e_a$ 
enters the critical interval, w.h.p. it does not exit the interval at the top. The use of critical intervals in a similar context was first introduced in \cite{BFL10}.

Let $\mathcal{F}_i$ be the natural filtration of the process (so conditioning on $\mathcal{F}_i$ tells us the values of all the variables, among other things).

For $i < T$, we have from $\eqref{1scA}$ that
\begin{align*}
&E[ \Delta A (i)| \mathcal{F}_i ] = -1 - \frac{3Y}{2M} - \frac{\zeta}{2M}  + \frac{2Z}{2M}  - 2 \cdot \frac{2Z}{2M} \cdot \frac{\zeta}{2M}  + O \of{ \frac{1}{M}}\\
&= -\frac{6A}{2B-A} + \frac{4\zeta(A+B)}{(2B-A)^2} +O \of{ \frac{1}{2B-A}+\frac{\zeta^2}{(2B-A)^2}}  \\
&= -\frac{6\of{na + \a + e_a}}{2(nb + e_b)-(na+\a+e_a)}+\frac{4\zeta\sqbs{\of{na + \a + e_a}+(nb + e_b)}}{\sqbs{2(nb + e_b)-\of{na + \a + e_a}}^2}+O 
\of{ \frac{1}{2B-A}+\frac{\zeta^2}{(2B-A)^2}}\\
&= -\frac{6a}{2b-a} + \frac{12ae_b - 12b \of{\a+ e_a}}{n(2b-a)^2} + \frac{4(a+b)\zeta}{n(2b-a)^2} +O \of{  \frac{1}{n(2b-a)}+ 
\frac{\a^2 + f_a^2 + f_b^2 + f_\zeta^2}{n^2(2b-a)^2}   } 
\end{align*}
The last equality follows from Lemma \ref{estlemma}.  Note that the lemma actually implies that the big-$O$ term includes mixed products of terms like 
$\a\cdot f_\z$ for example. We have simplified by using the fact that for all real numbers $x$ and $y$, $\abs{xy} \le\frac{1}{2}\of{x^2 + y^2}.$
We are now motivated to cancel out the $\zeta$ term in the last line by recursively defining 
\begin{align}
\a(0)&:=0 \\
\a(i+1)&:= \a(i) + \frac{4(a+b)\zeta -12b \a(i)}{n(2b-a)^2}. \label{eq:alphadef}
\end{align}

From this definition and the defintion of $f_\z$, it follows that for $i\le T$,
\begin{equation}\label{eq:alphabound}
0 \le \a(i) \le \sum_{j=0}^{i}\frac{4(a+b)f_\z}{n(2b-a)^2} \le C_\a \cdot
\begin{cases}
 \log n(1-t)^{-1/2}&\mbox{ for }i\le n-n^{3/5}\log^{2/5}n \\
 n^{1/5}\log^{9/5}n&\mbox{ for }n-n^{3/5}\log^{2/5}n < i \le T.
\end{cases}
\end{equation}  
as long as \begin{equation}\label{cond1}C_\a > 8 C_\z.\end{equation}

Now for $j \le i < T_j$, we have the supermartingale condition
\begin{align}
E[ \Delta A^+ (i)| \mathcal{F}_i ] &= E[ \Delta A (i)| \mathcal{F}_i ] - a'(t) - \displaystyle\frac{ 4(a+b)\z - 12b \a(i)}{n(2b-a)^2} - 
\frac{1}{n} f_a'(t)\nonumber\\
&\quad + O\left(  \frac{1}{n} a''(t) + \frac{1}{n^2} f_a''(t)\right)\\
& \le -\frac{  12 b g_a}{n(2b-a)^2} - \frac{1}{n} f_a'(t)\nonumber\\
&\quad +O\left( \frac{ a f_b }{n(2b-a)^2}+ \frac{1}{n(2b-a)}+\frac{\a^2 + f_a^2 + f_b^2 + f_\zeta^2}{n^2(2b-a)^2} +  
\frac{1}{n} a''(t) + \frac{1}{n^2} f_a''(t)\right) \label{eq:asmcondition}
\end{align}
Note that in the last line we have used $\eqref{eq:alphadef}$, the fact that $e_a \ge g_a$, and also that $a$ satisfies the differential equation $\eqref{diffeq}$. By taking $g_a=\Omega\of{f_a}$, we see that $A^+(j),\ldots,A^+(T_j)$ is a supermartingale since
\begin{align*}
 -\frac{  12 b g_a}{n(2b-a)^2} - \frac{1}{n} f_a' &= -\Omega\of{ n^{-1/2}\log^{1/2}n (1-t)^{-1/4}} 
\end{align*}
which dominates the big-$O$ term in $\eqref{eq:asmcondition}$.  

We use the following asymmetric version of the Azuma-Hoeffding inequality (for a proof see \cite{Bo10}):

\begin{lemma} Let $X_j$ be a supermartingale, such that  $-C \leq \Delta X(j) \leq c$ for all $j$, for $c < \frac{C}{10}$. Then for any $a < c m$ 
we have $Pr(X_m -X_0 > a) \leq \exp \left(- \frac{a^2}{3 c C m}\right)$ 
\end{lemma}

We have $$ -2 \le \Delta A \le 0$$
and $$- 2  \of{ 1-t }^\frac{1}{2} \le a'(t) \le 0.$$
This follows from analysis of the function $a(t)$.
So  $$ - 2  \le \Delta A^+ \le  2 \of{ 1-\frac{j}{n} }^\frac{1}{2} $$
for the supermartingale $A^+(j) \cdots A^+({T_j)}.$
Thus, if $A$ crosses its upper boundary at the stopping time $T$, then there is some step $j$ (with $T=T_j$) such that $$A^+(j) \le g_a \left(\frac{j}{n} \right) 
-  f_a \left(\frac{j}{n} \right)+ 2$$ and $A^+\of{T_j}>0$. In this case, $j$ is intended to represent the step when $e_a$ enters the crtical interval, 
$\eqref{eq:acrit}$.
 Applying the lemma we see that the probability of the supermartingale $A^+$ having such a large upward deviation has probability at most 
$$\exp\left\{ -  \displaystyle \frac{\of{f_a\of{\frac{j}{n}} - g_a\of{\frac{j}{n}} -2}^2}{ 12 n\of{1-\frac{j}{n}}^\frac{3}{2}   }  \right\}.$$
As there are $O\of{n}$ supermartingales $A^+ (j), \ldots, A^+ (T_j)$, we must choose $f_a, g_a$ to make the above probability $o\of{\frac{1}{n}}$. The following choice suffices:
$$f_a(t) =  C_A (1-t)^\frac{3}{4} n^\frac{1}{2} \log^\frac{1}{2} n$$
 $$g_a(t) = \frac{3}{4} f_a(t).$$
as long as the constant $C_A$ is chosen so that \begin{equation}\label{cond2}\frac{ \of{ \frac{1}{4} C_A}^2}{12} >1.\end{equation}

If we define \[A^-:=A -na - 
\a + f_a = e_a + f_a\] then we may prove that $A^-$ stays positive w.h.p. in a completely analogous fashion.

\subsection{$T$ is not triggered by $\z$}

Referring to $\eqref{zetatable}$, we may say that if $\z(i) > 0,$
\begin{align}
E[ \Delta \zeta (i)| \mathcal{F}_i ] &= p_z p_y - p_y + O \of{ p_\z} = -\frac{9A^2}{(2B-A)^2} + O \of{\frac{\zeta}{2B-A}}. \label{E1sczeta}
\end{align}
Now, before $T$ we have
\begin{align}
 \frac{9a^2}{(2b-a)^2}-\frac{9A^2}{(2B-A)^2}   &=  -9\of{ \frac{A}{2B-A} - \frac{a}{2b-a} } \of{ \frac{A}{2B-A} + \frac{a}{2b-a} }\nonumber\\
&= -9\left[ \frac{2b(\a+e_a) -2ae_b}{n(2b-a)^2} + O \of{ \frac{\a^2 + f_a^2 + f_b^2}{n^2(2b-a)^2} } \right]\nonumber\\
&\qquad\times\left[ 2 \of{ \frac{a}{2b-a} } + \frac{2b(\a+e_a) -2ae_b}{n(2b-a)^2} + O \of{ \frac{\a^2 + f_a^2 + f_b^2}{n^2(2b-a)^2} } \right]\nonumber\\
&=  \frac{36a(ae_b-b\a-be_a)}{n(2b-a)^3} + O\left(\frac{\a^2 + f_a^2 + f_b^2}{n^2(2b-a)^2} \right)\label{eq:phierror}.
\end{align}
In the last step we have cleaned up the big-$O$ using the facts $$\frac{\a + f_a+f_b}{n(2b-a)} = o(1)\qquad \textrm{ and }
\qquad\frac{a}{2b-a} = O(1).$$
For every step $j$, we define a stopping time $$T_j:= \min \left\{  i(j),\max(j, T) \right\}$$ where $i(j)$ is the least index $i \ge j$ such that $\zeta(i)=0$.  
Also, define a sequence $\zeta^+_j (j) \cdots \zeta^+_j(T_j)$, where

$$\zeta^+_j(i) := \zeta(i) + \displaystyle \sum_{j \le k < i} \Phi \left( \frac{k}{n} \right) - h_j\left( \frac{i}{n} \right)$$
where $h_j$ is some function we will choose that will make $\zeta^+_j(i)$ a supermartingale. 
Now for $j \le i < T_j$, using $\eqref{eq:phierror}$, we have
\begin{align}
E[ \Delta \zeta_j^+ (i)| \mathcal{F}_i ] &=-\frac{9A^2}{(2B-A)^2}  + \frac{9a^2}{(2b-a)^2} - \frac{1}{n} h_j'(t)    + 
O\of{\frac{\zeta}{2B-A}+ \frac{1}{n^2} h_j''(t)} \\
&\le    \frac{36a(af_b+bf_a)}{n(2b-a)^3} - \frac{1}{n} h_j'(t)+O\of{\frac{\a^2 + f_a^2 + f_b^2}{n^2(2b-a)^2} + \frac{f_\zeta}{n(2b-a)} +\frac{1}{n^2} h_j''(t) }.
\label{superzeta}
\end{align}
Note that
 \begin{equation} \label{eq:zetaplusub}
  \frac{36a(af_b+bf_a)}{(2b-a)^3}  \le  \of{\frac{9}{8} C_A + o(1)} n^{\frac{1}{2}} \log^\frac{1}{2} n \of{1-t}^\frac{1}{4}
 \end{equation}
so the choice
$$h_j(t) := C_h \of{1-\frac{j}{n}}^\frac{1}{4} n^\frac{1}{2} \log^\frac{1}{2} n  \of{t-\frac{j}{n}}$$
 makes the sequence a supermartingale as long as the constant $C_h$ is chosen so that \begin{equation}\label{cond3}C_h > \frac{9}{8} C_A.\end{equation}
Since $h_j\of{\frac{j}{n}}=0$, we will always have $\zeta^+_j(j) = \zeta(j)$.

We'll use the following supermartingale inequality due to Freedman \cite{F75}:
\begin{lemma}\label{lem:Freedman}
Let $X_i$ be a supermartingale, with $\Delta X_i \leq C$ for all $i$, and $V(i) :=\displaystyle \sum_{k \le i} Var[\Delta X(k)| \mathcal{F}_{k}]$  Then
$$P\left[\exists i: V(i) \le v, X_i - X_0 \geq d \right] \leq \displaystyle \exp\left(-\frac{d^2}{2(v+Cd) }\right).$$ \end{lemma}
Referring to $\eqref{zetatable}$, before $T$ we can put 
\begin{align*}
Var[\Delta \zeta^+_j (i) | \mathcal{F}_i] &= Var[ \Delta \zeta (i) | \mathcal{F}_i]\\
&\le E\sqbs{\of{\Delta\z(i)}^2 \mid \scr{F}_i}\\
&=1\cdot p_zp_y + 1\cdot p_y + 4\cdot(p_\z+p_zp_\z) + O\of{\frac{1}{M}}\\ 
&\le 3p_y
\end{align*}
and note that before $T$, we have
\begin{equation}\label{eq:pyub}
p_y = \frac{3Y}{2M} \le \frac{3A}{2B-A} \le \frac{3[n(1-t)^\frac{3}{2} + \a + f_a]}{4n(1-t) -2f_b - n(1-t)^\frac{3}{2} -\a -f_a} \le \of{1+ \frac{C_\a}{C_T^\frac{3}{2}} + o(1)}(1-t)^{\frac{1}{2}}
\end{equation}
so we will just say $p_y \le C_{p_y}(1-t)^{\frac{1}{2}}$ for some constant $C_{p_y}$ such that 
\begin{equation}\label{cond4}C_{p_y} > 1+ \frac{C_\a}{C_T^\frac{3}{2}}.\end{equation}
 Also, note that $\Delta \zeta^+ \le 2$.

Suppose the variable $\zeta$ triggers the stopping time $T$. Then there are steps $j < i=T$ such that  $\z>0$ all the way from step $j$ to step $i$, and 
$\zeta^+_{j}(i)>\ell_j(t) - h_j(t)$. We'll need to apply the lemma to the supermartingale $\zeta^+_j$ to show this event has low probability (guiding our 
choice for $\ell_j$). 
Note that in the lemma we can plug in the following for $v$:
\begin{align*}
V(i) &=\displaystyle \sum_{j \le k \le i} Var[\Delta \zeta^+_j (k)  | \mathcal{F}_{k}]\le 3 C_{p_y} \of{1-\frac{j}{n}}^{\frac{1}{2}}(i-j).
\end{align*}
So the unlikely event has probability at most

$$\exp\left\{ -  \displaystyle \frac{(\ell_j - h_j)^2}{ 2 \left[ 3 C_{p_y} \of{1-\frac{j}{n}}^{\frac{1}{2}}(i-j) + 2(\ell_j - h_j) \right]   }\right\}. $$

As there are $O\of{n^2}$ pairs of steps $j, i$ we'd like to make the above probability $o\of{\frac{1}{n^2}}$. Towards this end we consider 2 cases.

If $   \left(1-\frac{j}{n} \right)^\frac{1}{4}  \of{i-j}^\frac{1}{2} \le \log^\frac{1}{2} n$, then it suffices to put $\ell_j - h_j = C_\ell \log n$ as long as 
\begin{equation}\label{cond5}\frac{C_\ell^2}{ 6C_{p_y} + 4C_\ell} >2.\end{equation}

If $   \left(1-\frac{j}{n} \right)^\frac{1}{4}  \of{i-j}^\frac{1}{2} > \log^\frac{1}{2} n$, then it suffices to put $\ell_j - h_j = C_\ell \left(1-\frac{j}{n} \right)^\frac{1}{4}  \of{i-j}^\frac{1}{2}\log^\frac{1}{2} n $. Thus we choose
$$\ell_j(t) :=  h_j(t) + C_\ell \max \left\{\log n,   \left(1-\frac{j}{n} \right)^\frac{1}{4}  \of{i-j}^\frac{1}{2}\log^\frac{1}{2} n \right\}. $$
With this choice,  w.h.p. $T$ is not triggered by $\zeta$.

\subsection{An upper bound on $\zeta$}

In this section we'll motivate our choice of the function $f_\z$.

\begin{lemma} \label{zetalem1}
W.h.p.  for all $j < n - 2 C_{x}^\frac{2}{5} n^{\frac{3}{5}} \log^\frac{2}{5} n $ such that $\zeta(j-1)=0$, we have
\begin{enumerate}
\item $\zeta(j')=0$ for some $j \le j' \le j + C_x \of{1-\frac{j}{n}}^{-\frac{3}{2}} \log n$, and

\item$\zeta(i) \le 2 C_\ell^2 \of{1-\frac{j}{n}}^{-\frac{1}{2}} \log n$ for all $j \le i \le j'-1$
\end{enumerate}
\end{lemma}

\begin{proof}
Suppose $\zeta(j-1) = 0$. Note that we then have $\zeta(j) \le 2$. $\Phi(t)/(1-t)$ is decreasing since
\begin{align*}
 \frac{d}{dt}\of{\frac{\Phi(t)}{1-t}} &= 2\of{\frac{3a}{2b-a}}\of{\frac{(2b-a)\cdot3\of{-\frac{6a}{2b-a}} - 3a\of{-4+\frac{6a}{2b-a}}}{(2b-a)^2(1-t)}} + 
(1-t)^{-2}\of{\frac{3a}{2b-a}}^2\\
&=-\frac{9a^3(8b-a)}{(1-t)^2(2b-a)^4} \le 0.
\end{align*}
Also, using $\eqref{eq:alimit}$ and the defintion of $b$,
\[\lim_{t\rightarrow 1^-}\frac{\Phi(t)}{1-t} = \frac{1}{6}.\]
Hence $\Phi(t) \ge \frac{1}{6} \of{1-t}$ for all $0\le t\le 1$. If we substitute $x = \frac{i-j}{n}$ then

$$\displaystyle \sum_{j \le k < i} \Phi\left(\frac{k}{n} \right) \ge \frac{1}{6}\of{1-\frac{i+j-1}{2n}}(i-j) \ge -\frac{1}{12}nx^2 + \frac{1}{6}n
\of{1-\frac{j}{n}}x. $$ Plugging in the value of $ \ell_{j}(t)$, we have that for any $i \ge j$ such that $\z(j) \ldots \z(i)$ are all positive,
\begin{align} 
 \zeta(i) &\le  \zeta(j) - \displaystyle \sum_{j \le k < i} \Phi\left(\frac{k}{n} \right)+ \ell_j(t) \\
 &\le \frac{1}{12}nx^2 - \left[\frac{1}{6}n\of{1-\frac{j}{n}} - C_h  n^\frac{1}{2} \log^\frac{1}{2} n \of{1 - \frac{j}{n} }^\frac{1}{4} \right] 
x\label{zetabound}\\
 &\quad +C_\ell \max \left\{\log n,   \left(1-\frac{j}{n} \right)^\frac{1}{4}  n^\frac{1}{2}\log^\frac{1}{2} n x^\frac{1}{2}\right\}  + 2. \nonumber
 \end{align}
Consider $\eqref{zetabound}$ for $x=x_j:= C_x n^{-1} \log n \of{1-\frac{j}{n}}^{-\frac{3}{2}}$. As long as $C_x>1$ and $j < n - 2 C_{x}^\frac{2}{5} n^{\frac{3}{5}} \log^\frac{2}{5} n $, we have
$$ \left(1-\frac{j}{n} \right)^\frac{1}{4} n^\frac{1}{2}\log^\frac{1}{2} n x_j^\frac{1}{2}> \log n $$
so we can evaluate the ``max'' in $\ell_j$.
Also note that the coefficient of $x$ is dominated by $-\frac{1}{6}n\of{1-\frac{j}{n}}$, so the coefficient of $x$ is at most, say $-\frac{1}{7}n\of{1-\frac{j}{n}}$.
Thus $\eqref{zetabound}$ gives
\begin{align*}
 \zeta(j+nx_j) \le \frac{C_x^2}{12} n^{-1} \log^2 n \of{1-\frac{j}{n}}^{-3} - \of{\frac{C_x}{7} - C_\ell  \sqrt{C_x} } \log n 
\of{1-\frac{j}{n}}^{-\frac{1}{2}} +2 
 \end{align*}
which is negative for this range of $j$ as long as we pick $C_x$ such that 
\begin{equation}\label{cond6} \frac{C_x}{7} - C_\ell  \sqrt{C_x} >0.\end{equation} 
Therefore, $\zeta$ must have hit $0$ again before step $i=j+nx_j$. This proves the first part of the lemma.

To prove the second part, consider $\eqref{zetabound}$ for $j < i < j+nx_j$ (i.e. for $0<x<x_j$). If $x \le n^{-1} \log n \of{1-\frac{j}{n}}^{-\frac{1}{2}}$ then we can put
\begin{align*}
 \zeta(i)  &\le \frac{1}{12}nx^2 - \frac{1}{7}n\of{1-\frac{j}{n}}  x +C_\ell \log n < 2 C_\ell \log n
 \end{align*}
 and for $x$ larger than that, we'll put
 \begin{align*}
 \zeta(i) &\le \frac{1}{12}nx^2 - \frac{1}{7}n\of{1-\frac{j}{n}} x +C_\ell \left(1-\frac{j}{n} \right)^\frac{1}{4} n^\frac{1}{2}\log^\frac{1}{2} n x^\frac{1}{2} \\
 & \le \frac{C_x^2}{12} n^{-1} \log^2 n \of{1-\frac{j}{n}}^{-3} + \frac{7C_\ell^2}{4} \of{1-\frac{j}{n}}^{-\frac{1}{2}} \log n\\
 &< 2 C_\ell^2 \of{1-\frac{j}{n}}^{-\frac{1}{2}} \log n.
 \end{align*}
where to justify the second line we use the inequality $c\sqrt{x} - dx \le \frac{c^2}{4d}$ for real numbers $x,c,d>0.$ 
\end{proof}
We would also like to say something about $\z(i)$ for $i > n - 2 C_{x}^\frac{2}{5} n^{\frac{3}{5}} \log^\frac{2}{5} n$.
\begin{lemma}
There exists a constant $C_\z$ such that w.h.p. for all $i \le T$ we have $\z(i) \le C_\z n^\frac{1}{5} \log^\frac{4}{5} n. $
\end{lemma}
\begin{proof}
Suppose step $j' \ge n - 2 C_{x}^\frac{2}{5} n^{\frac{3}{5}} \log^\frac{2}{5} n$ with $\z(j')=0$. It follows from Lemma \ref{zetalem1} that w.h.p. such a $j'$ exists.
 Let $i \ge j'$  such that $\z(j') \ldots \z(i)$ are all positive. Note that we again have the bound $\eqref{zetabound}$. 
But now $0 \le x \le \frac{n-j'}{n} \le 2 C_{x}^\frac{2}{5} n^{-\frac{2}{5}} \log^\frac{2}{5} n$, and $\eqref{zetabound}$ gives
$\z(i) \le \of{\frac{1}{3}C_x^\frac{4}{5} + 2^\frac{3}{4} C_\ell C_x^\frac{3}{10}}   n^\frac{1}{5} \log^\frac{4}{5} n.$
\end{proof}

So in particular we can say that for $i \le T$ we have $$\z(i) \le f_\zeta (t) = C_\z \min \left\{ (1-t)^{-\frac{1}{2}} \log n, n^\frac{1}{5} \log^\frac{4}{5} n \right\},$$ where 
\begin{equation}\label{cond7}C_\z > \max \left\{ 2C_\ell^2, \frac{1}{3}C_x^\frac{4}{5} + 2^\frac{3}{4} C_\ell C_x^\frac{3}{10} \right\}\end{equation}

\subsection{$T$ is not triggered by $B$}\label{sec:notB}
 
 Recall from $\eqref{eq:Btrajectory}$ that $$e_b(i) = \sum_{j \le i} \of{1_{\s(i)=loop} - 1_{\delta(j)=\zeta} }.$$
 First we'll bound $ \displaystyle \sum_{j\le i}  1_{\delta(j) = \zeta}$.
 Define $B^-(i) := -\displaystyle \sum_{j\le i} 1_{\delta(j) = \zeta} + \frac{1}{2}f_b(t)$. Then 
 \begin{align*}
E[ \Delta B^- (i)| \mathcal{F}_i ] &= -\frac{2Z}{2M} \cdot \frac{\zeta}{2M-2} + \frac{1}{2n} f_b ' (t) + O \of{\frac{1}{n^2} f_b '' (t) }\\
& \ge- \frac{ f_\zeta}{n(2b-a)} + \frac{1}{2n} f_b ' (t) + O \of{\frac{1}{n^2} f_b '' (t) }.\\
\end{align*}
Note that by $\eqref{eq:abounds}$, $$3(1-t) \le 2b-a \le 4(1-t),$$
so we can put 
\begin{equation}
f_b = C_B \cdot \begin{cases}
 (1-t)^{-\frac{1}{2}} \log n&: 1-t > n^{-\frac{2}{5}} \log^\frac{2}{5} n \\
-n^\frac{1}{5} \log^\frac{4}{5} n \log(1-t)&: \textrm{otherwise} \\
\end{cases}
\end{equation}
 and $B^-$ will be a submartingale as long as 
\begin{equation}\label{cond8}C_B > \frac{4}{3} C_\z.\end{equation} 
We'll apply Lemma \ref{lem:Freedman} to $-B^-$. Note that before $T$ we can put 
\begin{align*}
Var[\Delta B^- (i) | \mathcal{F}_i] &= Var[ 1_{\delta(i) = \zeta} | \mathcal{F}_i]\\
&\le p_\zeta \le  \frac{f_\z}{4n(1-t) - na -\a -f_a - f_b} \le \frac{f_\z}{3n(1-t)}
\end{align*}
and therefore, referring to $V(i)$ as in 
Lemma \ref{lem:Freedman}, 
\begin{align*}
V(i) &\le\displaystyle \sum_{0 \le k \le i}  \frac{f_\z}{3n(1-t)}.
\end{align*}
So for $v$ we will plug in
\begin{equation}
v =  C_{v_B} \cdot 
\begin{cases}
 (1-t)^{-\frac{1}{2}} \log n&: 1-t > n^{-\frac{2}{5}} \log^\frac{2}{5} n \\
n^\frac{1}{5} \log^\frac{9}{5} n&: \textrm{otherwise}. \\
\end{cases}
\end{equation}
which is an upper bound on $V(i)$ as long as 
\begin{equation}\label{cond9}C_{v_B} \ge \frac{2}{3} C_\z\end{equation}
Note $\left|\Delta B^- \right| \le 1$, so the probability that $-B^-(i) > \frac{1}{2}f_b(t)$  is at most
$$\exp\left\{-\frac{\frac{1}{4} f_b^2}{2 \left[ v + \frac{1}{2}f_b \right]}   \right\} $$
which is $ o\of{\frac{1}{n}}$ as long as \begin{equation}\label{cond10}\frac{ \frac{1}{4} C_B^2}{2[C_{v_B} + \frac{1}{2} C_B]} > 1.\end{equation}
So w.h.p. for all $i \le T$, we have $$\displaystyle \sum_{j<i} 1_{\delta(j) = \zeta} \le f_b(t).$$
The sum $\displaystyle \sum_{j<i} 1_{\s(j)=loop}$ presents less difficulty, since w.h.p. the configuration has at most $C_B \log n$ loops total. So we can 
trivially say that 
$$\sum_{j<i} 1_{\s(j)=loop} \le f_b(t)$$
and hence w.h.p. the stopping time $T$ is not triggered by variable $B$.

\subsection{Values for the constants}\label{sec:kappavalues}
Throughout the proof above, we collect various constraints on the constants in \eqref{cond1}, \eqref{cond2}, \eqref{cond3}, \eqref{cond4}, \eqref{cond5}, \eqref{cond6}, \eqref{cond7}, \eqref{cond8}, \eqref{cond9} and \eqref{cond10}. The reader may chack that the following values satisfy all the conditions.
\[C_A=16,\quad C_h=20,\quad C_{p_y}=2,\quad C_\ell=12,\quad C_x=8000,\]
\[C_\z=800,\quad C_\a=70000,\quad C_{v_B}=700,\quad C_B=1200,\quad C_T=2000.\]
This completes the proof of Theorem \ref{Ttheorem}.

\section{Upper bound on the number of components}

In this section we prove the following lemma which provides the upper bound for the proof of Theorem \ref{thm:main}: 

\begin{lemma}
W.h.p. the algorithm outputs a $2$-matching with $O\of{n^\frac{1}{5} \log^\frac{9}{5} n}$ components.
\end{lemma}

\begin{proof}

The components of our $2$-matching at any step $i$ consist of cycles and paths (including paths of length $0$). First we'll bound the number of paths in the 
final $2$-matching. Note that these final paths have both endpoints in $Z_0$, meaning that each endpoint had a half-edge in $\z$ that got deleted (or for
 paths of length $0$ there is only one vertex which is in $Y_0$). So to bound the number of these paths, we bound the sum $\displaystyle \sum_{j}  
1_{\delta(j) = \zeta}$.
Note that in light of Section \ref{sec:notB}, we have the bound 

$$\displaystyle \sum_{j<T}  1_{\delta(j) = \zeta} = O\of{n^\frac{1}{5} \log^\frac{9}{5} n }.$$

Next we'll bound the terms corresponding to steps after $T$, but before $A=0$. By Theorem \ref{Ttheorem} we have w.h.p.
$$A(T) = O\of{n^\frac{1}{5} \log^\frac{9}{5} n }$$
since \[0 \le \a(T) =  O\of{n^\frac{1}{5} \log^\frac{9}{5} n}\]
by $\eqref{eq:alphabound}$, and \[na\of{\frac{T}{n}},\,f_a\of{\frac{T}{n}} = O\of{n^{1/5}\log^{9/5}n}.\]
Now note that by $\eqref{1scA}$, on each step $j$ such that $\s(j)\in \braces{Z,multi}$ and $\d(j) = \z$, the variable $A$ decreases by $2$. Also, the 
variable $A$ is nonincreasing. Therefore there can be at most $O\of{n^\frac{1}{5} \log^\frac{9}{5} n }$ such steps $j$ until $A=0$. 

Once we have $A=0$, the algorithm finds a maximum matching on the remaining random $2$-regular graph $\G$. Thus, to complete the bound on the number of 
paths in the final $2$-matching, we'll bound the number of vertices in $\G$ that are unsaturated by the matching (i.e. the number of odd cycles in the 
remaining $2$-regular graph $\G$). But $\G$ has at most $O\of{\log n}$ cycles total, since it's a random $2$-regular graph. 
Thus, the sum $\displaystyle \sum_{j}  1_{\delta(j) = \zeta}$, and therefore the number of paths in the final $2$-matching, are 
$O\of{n^\frac{1}{5} \log^\frac{9}{5} n }.$

Now we bound the number of cycles in the final $2$-matching. Note that at any step, the probability of closing a cycle is at most $\frac{1}{2M-1}$. 
Therefore, the number of cycles created for the whole process is stochastically dominated by the random variable 
$$C:= \sum_{j=1}^{3n} c_j$$
where
\begin{equation}
c_j =  
\begin{cases}
 1&:  \mbox{with prob. } \frac{1}{j}\\
0&: \mbox{with prob. } \frac{j-1}{j}. \\
\end{cases}
\end{equation}
So if we define the martingale $$C(i) := \sum_{j=1}^{i} \of{c_j - \frac{1}{j}}$$
then we have $ Var[\Delta C(i)] = \frac{i-1}{i^2}$, and note $\sum_{i=1}^{3n} \frac{i-1}{i^2} = O(\log(3n))$.
Now, applying Lemma \ref{lem:Freedman} to $C(i)$ shows that w.h.p. it is always at most $O(\log^\frac{1}{2} n)$, and since $E[C] = O(\log n)$, we 
have that $C= O(\log n)$ w.h.p..

\end{proof}

\begin{section}{Lower bound on the number of components}
 In this section we will prove that near the end of the process, there is a non-zero probability that $\z$ becomes large and stays large for a significant 
amount of time. In this case, the algorithm will likely delete an edge adjacent to a $\z$ vertex. In particular, we will prove the following lemma which 
provides the lower bound and thus completes the proof of Theorem \ref{thm:main}:
\begin{lemma}
 W.h.p. the algorithm outputs a 2-matching with $\Omega\of{n^{\frac{1}{5}}\log^{-4}n}$ components.
\end{lemma}

\begin{proof}
We show that $\z$ stochastically dominates a suitably defined martingale and then apply the following central limit theorem of Freedman.

\begin{lemma}\label{FreedmanCLT}
 Let $S_i$ be a martingale adapted to the filtration $\scr{F}_i$ with $X_i := S_{i} - S_{i-1}$, $\abs{X_i} \le C$ for some constant $C$, and let 
$V_i := \sum_{k\le i}Var\sqbs{X_k | \scr{F}_{k-1}}$. For each $n$, let $0 < \g_n < \g'_n$ be real numbers, and let $\s_n$ be a stopping time. 
As $n\rightarrow\infty$, suppose $\g_n\rightarrow\infty$ and $\g'_n/\g_n\rightarrow 1$ and $\Pr{\g_n < V_{\s_n} < \g'_n}\rightarrow 1$.  
Then $S_{\s_n}/\sqrt{\g_n}$ converges in distribution to $\scr{N}(0,1).$
\end{lemma}

Let \[w(i)=\frac{3a(i/n)}{2b(i/n) - a(i/n)}.\]  In this section we will consider steps from $i_0 = n - n^{3/5}$ to $i_{end} = n- n^{3/5} + 
n^{3/5}\log^{-1}n \le n-\frac{1}{2}n^{3/5}$. From Theorem \ref{Ttheorem}, w.h.p., $T$ occurs after this time frame. Hence we have dynamic 
concentration on our variables and can say in this range,
\begin{align}
 p_y(i) &= w(i) - O(n^{-2/5}\log n) \\
 p_\z(i) &= O(n^{-2/5}\log n) \\
 p_z(i) &= 1 - p_y(i) - p_\z(i).
\end{align}
Note that in this range we also have $w(i) = \Theta(n^{-1/5})$.
Our martingale will have independent increments given by
\begin{equation}
 X(i)=
\begin{cases}
1&\mbox{with prob. } w(i) - Ln^{-2/5}\log n\\
0&\mbox{with prob. } 1 - 2w(i) - Ln^{-2/5}\log n\\
-1&\mbox{with prob. } w(i) - Ln^{-2/5}log n\\
-2&\mbox{with prob. } 3Ln^{-2/5}\log n
\end{cases}
\end{equation}
where $L$ is a positive contstant large enough that for all $i_0\le i \le i_{end}$
\[p_z(i)p_y(i) \ge w(i) - Ln^{-2/5}\log n,\quad p_z(i)p_y(i) + p_z(i)^2 \ge 1- w(i) - 2Ln^{-2/5}\log n\]
and
\[p_z(i)p_y(i)+p_z(i)^2 +p_y(i) \ge 1-3Ln^{-2/5}\log n.\]
In this case, $\D\zeta(i)$ stochastically dominates $X(i)$. This follows from (\ref{zetatable}) in the case when $\z>0$ and trivially when $\z=0$. 

For any $i_0 < i \le i_{end}$ we have 
\[E[X(i) | \scr{F}_{i-1}] = -6Ln^{-2/5}\log n\]
and \[Var[X(i) | \scr{F}_{i-1}] = 2w(i) + 4Ln^{-2/5}\log n.\] 

We will split the time range $i_0$ to $i_{end}$ into $d=\log n$ many chunks of length $n^{3/5}\log^{-2} n$.
Recall that $i_0 = n-n^{3/5}$ and for all $1\le\ell\le d$ define
\[i_\ell = i_{\ell-1} + n^{3/5}\log^{-2}n.\]

For $0\le \ell < d$, we define a martingale starting at $i_\ell$ to be
\[S_{\ell}(k) = \sum_{i=i_\ell+1}^{k}\of{X(i) - E[X(i) | \scr{F}_{i-1}]}.\]
Then for $0\le \ell < d$ we have
\[S_{\ell}(i_{\ell+1}) =\of{\sum_{i=i_{\ell}+1}^{i_{\ell+1}}X(i)} + 6Ln^{1/5}\log^{-1}n.\]
We also have that
\[V_{\ell}:=\sum_{i=i_{\ell}+1}^{i_{\ell+1}}Var[X(i)|\scr{F}_{i-1}] = \Theta\of{n^{2/5}\log^{-2} n}.\]
Further, using the fact that for  the expression for $V_{\ell}$ is completely deterministic, we may choose $\g(n,\ell)$ such that 
$\g(n,\ell)< V_{\ell} < \g(n,\ell) + o(\g(n,\ell))$. 
By using the facts about $a(t)$ and $b(t)$ presented in Section \ref{sec:expbehavior}, we may take $\g(n,\ell) = C_\ell n^{2/5}\log^{-2} n$ for come 
constant $C_\ell$. Note here that there is an absolute constant $c$ such that $C_\ell \le c$ for all $0\le \ell < d$. 

Hence applying Lemma \ref{FreedmanCLT} to $S_\ell$ with stopping time $i_{\ell+1}$, we see that 
\[\frac{\of{\sum_{i=i_{\ell}+1}^{i_{\ell+1}}X(i)} + 6Ln^{1/5}\log^{-1}n}{\sqrt{C_\ell n^{2/5}\log^{-2} n}} \stackrel{d}{\rightarrow} \scr{N}(0,1).\]
So there exists some constant $p_0 > 0$ such that for each $0\le \ell < d$ (and $n$ sufficiently large),
\[\Pr{\frac{\of{\sum_{i=i_{\ell}+1}^{i_{\ell+1}}X(i)} + 6Ln^{1/5}\log^{-1}n}{\sqrt{C_\ell n^{2/5}\log^{-2} n}} \ge \frac{6L+1}{\sqrt{c}}} \ge p_0\]
So we get that 
\begin{align*}
\Pr{\forall\, 0< \ell \le d,\, \zeta(i_\ell) \le n^{1/5}\log^{-1} n} &\le\Pr{\forall\, 0\le \ell < d,\,\sum_{i=i_\ell+1}^{i_{\ell+1}}X(i) \le n^{1/5}\log^{-1}n}\\
 &\le \of{1-p_0}^{\log n}\\
 &=o(1).
\end{align*}
So we know that w.h.p. there is a point $i_b$ where $\z(i_b) > n^{1/5}\log^{-1}n$. We would like to show that after $n^{3/5}\log^{-3}n$ steps, $\z$ has 
not decreased below $\frac{1}{2}n^{1/5}\log^{-1}n$. To prove this, we consider the martingale \[S_b(k) = n^{1/5}\log^{-1}n + \sum_{i=i_b}^{k}\of{X(i) - E[X(i) | 
\scr{F}_{i-1}]}.\]
Let $i_c = i_b + n^{3/5}\log^{-3}n$.
Then \[\sum_{i=i_b+1}^{i_c}Var[X(i)|\scr{F}_{i-1}] = \Theta\of{n^{2/5}\log^{-3}n}\]
By applying Lemma \ref{lem:Freedman} to this martingale, we have that after $n^{3/5}\log^{-3}n$ steps, 
\begin{align*}
\Pr{\exists i\,:\,i_b \le i\le i_c,\, \z(i) \le \frac{1}{2}n^{1/5}\log^{-1}n}&\le \Pr{\exists i\le i_c \,:\, S_b(i) \le \frac{1}{2}n^{1/5}\log^{-1}n}\\
	&\le \exp\of{-\Omega\of{\frac{n^{2/5}\log^{-2}n}{n^{2/5}\log^{-3}n\of{1+n^{-1/5}\log^{-2}n}}}}\\
	&\le o(1).
\end{align*}
So we know that whp, $\z(i) \ge \frac{1}{2}n^{1/5}\log^{-1}n$ for $i_b \le i \le i_c$. In this time, the algorithm is likely to delete an edge adjacent 
to a $\z$ vertex. Formally, we have that there exists some $q_0$ such that 
for all $i_b \le i \le i_c$, \[p_z(i)p_\z(i) \ge q_0 = \Omega\of{n^{-2/5}\log^{-1}n}\]
so that if $W$ is a random variable representing the number of $i$ between $i_b$ and $i_c$ when $\d(i) = \z$, then $W$ stochastically dominates 
$\Bin(n^{3/5}\log^{-3}n, q_0)$.
\[E[\Bin(n^{3/5}\log^{-3}n, q_0)] = \Omega\of{n^{1/5}\log^{-4}n},\]
so an application of the Chernoff bound tells us that, w.h.p., $W=\Omega\of{n^{1/5}\log^{-4}n}.$
\end{proof}

\end{section}

\end{document}